\documentclass[reqno]{amsart}
\usepackage{amsmath, amsthm, amssymb, amstext}

\usepackage{hyperref,xcolor}% http://ctan.org/pkg/{hyperref,xcolor}
\hypersetup{
 pdfborder={0 0 0},
 colorlinks,
}

\usepackage{enumitem}
\newtheorem{theorem}{Theorem}
\newtheorem{remark}[theorem]{Remark}
\newtheorem{lemma}[theorem]{Lemma}
\newtheorem{proposition}[theorem]{Proposition}

\DeclareMathOperator*{\divergenz}{div}              %
\DeclareMathOperator*{\ints}{int}         %
\DeclareMathOperator*{\essinf}{ess ~inf}         %
\DeclareMathOperator*{\ww}{w}         %
\DeclareMathOperator*{\Ss}{S}         %

\newcommand{\N}{\mathbb{N}}

\newcommand{\R}{\mathbb{R}}

\newcommand{\Lp}[1]{L^{#1}(\Omega)}
\newcommand{\Lpvalued}[1]{L^{#1}\left(\Omega;\R^N\right)}

\newcommand{\Wpzero}[1]{W^{1,#1}_0(\Omega)}

\newcommand{\eps}{\varepsilon}
\newcommand{\ph}{\varphi}

\newcommand{\into}{\int_{\Omega}}
\newcommand{\weak}{\overset{\ww}{\to}}

\newcommand{\Linf}{L^{\infty}(\Omega)}
\newcommand{\close}{\overline{\Omega}}
\newcommand{\interior}{\ints \left(C^1_0(\overline{\Omega})_+\right)}
\newcommand{\cprime}{$'$}

\renewcommand{\l}{\left}
\renewcommand{\r}{\right}

\numberwithin{theorem}{section}
\numberwithin{equation}{section}

\title[Positive solutions for weighted singular $p$-Laplace equations]{Positive solutions for weighted singular $p$-Laplace equations via Nehari manifolds}

\author[N.\,S.\,Papageorgiou]{Nikolaos S.\,Papageorgiou}
\address[N.\,S.\,Papageorgiou]{National Technical University, Department of Mathematics, Zografou Campus, Athens 15780, Greece}
\email{npapg@math.ntua.gr}

\author[P.\,Winkert]{Patrick Winkert}
\address[P.\,Winkert]{Technische Universit\"{a}t Berlin, Institut f\"{u}r Mathematik, Stra\ss e des 17.\,Juni 136, 10623 Berlin, Germany}
\email{winkert@math.tu-berlin.de}


\begin{document}
%\vspace*{-2cm} \begin{center}\bf \today\end{center} \vspace{2cm}

\begin{abstract}
      In this paper we study weighted singular $p$-Laplace equations involving a bounded weight function which can be discontinuous. Due to its discontinuity classical regularity results cannot be applied. Based on Nehari manifolds we prove the existence of at least two positive bounded solutions of such problems.
\end{abstract}

\maketitle

%***********************************************************************************************************************************
%***********************************************************************************************************************************
\section{Introduction}%\label{section_introduction}
%***********************************************************************************************************************************
%***********************************************************************************************************************************

Let $\Omega \subseteq \R^N$, $N \geq 1$, be a bounded domain with a Lipschitz boundary $\partial \Omega$. In this paper, we study the following nonlinear singular Dirichlet problem
\begin{align}\tag{P$_\lambda$}\label{problem}
  \begin{split}
    &-\divergenz(\xi(x)|\nabla u|^{p-2}\nabla u)
    =a(x)u^{-\gamma} +\lambda u^{r-1}\quad \text{in } \Omega \\
    &u\big|_{\partial \Omega}=0, \quad 0<\gamma<1, \quad 1<p<r<p^*, \quad u \geq 0, \quad \lambda>0.
   \end{split}
\end{align}
In this problem the differential operator is a weighted $p$-Laplacian with a weight $\xi\in \Linf$, $\xi\geq 0$ and $\xi$ is supposed to be bounded away from zero. Since $\xi$ is discontinuous in general, we cannot use the nonlinear global regularity theory of Lieberman \cite{Lieberman-1991} and the nonlinear strong maximum principle, see Pucci-Serrin \cite[pp.\,111 and 120]{Pucci-Serrin-2007}. The fact that these two basic tools are no longer available leads to a different approach in the analysis of problem \eqref{problem} which is based on the Nehari method. On the right-hand side of \eqref{problem} we have the competing effects of two different nonlinearities. One is the singular term $s\to a(x) s^{-\gamma}$ with $s > 0$ and the other one is a parametric $(p-1)$-superlinear perturbation $s \to \lambda s^{r-1}$ with $s\geq 0$ and $p<r<p^*$ with $p^*$ being the critical Sobolev exponent corresponding to $p$ defined by
\begin{align*}
    p^*=
    \begin{cases}
	\frac{Np}{N-p} &\text{if }p <N,\\
	+\infty & \text{if } N \leq p.
    \end{cases}
\end{align*}

We are looking for positive solutions of problem \eqref{problem} and we show that problem \eqref{problem} has at least two positive solutions for all $\lambda \geq 0$.

Singular problems with such competition phenomena were investigated by Sun-Wu-Long \cite{Sun-Wu-Long-2001} and Haitao \cite{Haitao-2003} for semilinear equations driven by the Laplacian and by Giacomoni-Schindler-Tak\'{a}\v{c} \cite{Giacomoni-Schindler-Takac-2007}, Papageorgiou-Smyrlis \cite{Papageorgiou-Smyrlis-2015}, Papageorgiou-Winkert \cite{Papageorgiou-Winkert-2019} and Perera-Zhang \cite{Perera-Zhang-2005} for equations driven by the $p$-Laplacian. We also refer to the works of Leonardi-Papageorgiou \cite{Leonardi-Papageorgiou-2019a}, \cite{Leonardi-Papageorgiou-2019b}. In all the mentioned works the weight function $\xi$ is equal to one and so we can use the global elliptic regularity theory and the strong maximum principle. These tools are crucial in the proofs of the works above and are combined with variational methods and suitable truncation and comparison techniques. The regularity theory guarantees that the solutions are in $C^1_0(\close)$ and then the strong maximum principle, so-called Hopf theorem, implies that these solutions are in $\interior$ which is the interior of the positive order cone of $C^1_0(\close)$. 

Without these facts the proofs of the works above are no more valid. As we already indicated, in our setting, these results do not hold, so we need to employ a different approach.

%***********************************************************************************************************************************
%***********************************************************************************************************************************
\section{Preliminaries} %\label{section_hypotheses}
%***********************************************************************************************************************************
%***********************************************************************************************************************************

We denote by $\Wpzero{p}$ the usual Sobolev space with norm $\|\cdot\|$. By the Poincar\'{e} inequality we have
\begin{align*}
    \|u\|=\|\nabla u\|_p\quad\text{for all }u \in \Wpzero{p},
\end{align*}
where $\|\cdot\|_p$ denotes the norm of $\Lp{p}$ and $\Lpvalued{p}$, respectively. The norm of $\R^N$ is denoted by $|\cdot|$ and ``$\cdot$'' stands for the inner product in $\R^N$. 
%For $s \in \R$, we set $s^{\pm}=\max\{\pm s,0\}$ and for $u \in W^{1,p}_0(\Omega)$ we define $u^{\pm}(\cdot)=u(\cdot)^{\pm}$. It is well known that
% \begin{align*}
%     u^{\pm} \in W^{1,p}_0(\Omega), \quad |u|=u^++u^-, \quad u=u^+-u^-.
% \end{align*}
%By $|\cdot|_N$ we denote the Lebesgue measure on $\R^N$. 
By $p^*>1$ we denote the Sobolev critical exponent for $p$ defined by
\begin{align*}
    p^*=
    \begin{cases}
    \frac{Np}{N-p} & \text{if }p<N,\\
    +\infty & \text{if } N \leq p.
    \end{cases}
\end{align*}

Let $\xi\in\Linf$ with $0<\essinf_{\Omega}\xi$ and let $A\colon\Wpzero{p}\to W^{-1.p'}(\Omega)=\Wpzero{p}^*$ with $\frac{1}{p}+\frac{1}{p'}=1$ be defined by
\begin{align}\label{p-Laplace}
    \langle A(u), \ph\rangle=\into \xi(x)|\nabla u|^{p-2}\nabla u \cdot \nabla \ph \,dx \quad\text{for all }u,\ph\in\Wpzero{p}.
\end{align}
The next proposition states the main properties of this map and it can be found in Gasi{\'n}ski-Papageorgiou \cite[Problem 2.192, p.\,279]{Gasinski-Papageorgiou-2016}.

\begin{proposition}%\label{proposition_operator}
    The map $A:\Wpzero{p}\to W^{-1,p'}(\Omega)$ defined in \eqref{p-Laplace} is bounded, that is, it maps bounded sets to bounded sets, continuous, strictly monotone, hence maximal monotone and it is of type $(\Ss)_+$, that is,
    \begin{align*}
	u_n \weak u \text{ in }\Wpzero{p}\quad\text{and}\quad \limsup_{n\to\infty} \langle A(u_n),u_n-u\rangle \leq 0,
    \end{align*}
    imply $u_n\to u$ in $\Wpzero{p}$.
\end{proposition}

%***********************************************************************************************************************************
%***********************************************************************************************************************************
\section{Positive Solutions}%\label{section_Multiplicity_Theorem}
%***********************************************************************************************************************************
%***********************************************************************************************************************************

We suppose the following hypotheses related to problem \eqref{problem} throughout this paper.
\begin{enumerate}[leftmargin=1.2cm]
    \item[H$_0$:]
	$\xi, a\in \Linf$, \ \ $0<\xi_0 \leq \essinf_{\Omega}\xi$,\ \ $a(x) > 0$ for a.\,a.\,$x\in\Omega$.
\end{enumerate}

This hypothesis implies that the natural function space for the analysis of problem \eqref{problem} is the Sobolev space $\Wpzero{p}$.

Let $\ph_\lambda\colon \Wpzero{p}\to \R$ be the energy functional for problem \eqref{problem} defined by
\begin{align*}
    \ph_\lambda(u) = \frac{1}{p} \into \xi(x) |\nabla u|^p \,dx - \frac{1}{1-\gamma} \into a(x) |u|^{1-\gamma}\,dx-\frac{\lambda}{r} \|u\|_r^r.
\end{align*}
It is clear that $\ph_\lambda$ is not $C^1$. The corresponding Nehari manifold for this functional is given by
\begin{align*}
    N_\lambda=\l\{u\in\Wpzero{p}\colon \into \xi(x) |\nabla u|^p\,dx=\into a(x) |u|^{1-\gamma}\,dx+\lambda \|u\|_r^r,\ u\neq 0 \r\}.
\end{align*}
We decompose $N_\lambda$ into three disjoint parts
\begin{align*}
    N_\lambda^+ & = \l\{ u \in N_\lambda\colon (p+\gamma-1)\into \xi(x) |\nabla u|^p\,dx-\lambda (r+\gamma-1) \|u\|_r^r>0\r\},\\
    N_\lambda^0 & = \l\{ u \in N_\lambda\colon (p+\gamma-1)\into \xi(x) |\nabla u|^p\,dx=\lambda (r+\gamma-1) \|u\|_r^r\r\},\\
    N_\lambda^- & = \l\{ u \in N_\lambda\colon (p+\gamma-1)\into \xi(x) |\nabla u|^p\,dx-\lambda (r+\gamma-1) \|u\|_r^r<0\r\}.
\end{align*}
Note that $N_\lambda$ is much smaller than $\Wpzero{p}$ and contains the nontrivial weak solutions of \eqref{problem}. It is possible for $\ph_\lambda \big|_{N_\lambda}$ to exhibit properties which fail globally. One such property is identified in the next proposition.

\begin{proposition}\label{proposition_1}
    If hypotheses H$_0$ hold, then $\ph_\lambda\big|_{N_\lambda}$ is coercive.
\end{proposition}

\begin{proof}
    Let $u \in N_\lambda$. From the definition of the Nehari manifold we have
    \begin{align}\label{1}
	-\frac{1}{r} \into \xi(x) |\nabla u|^p \,dx +\frac{1}{r} \into a(x) |u|^{1-\gamma} \,dx=-\frac{\lambda}{r} \|u\|^r_r.
    \end{align}
    From \eqref{1} and hypotheses H$_0$ we obtain
    \begin{align}\label{2}
      \begin{split}
	\ph_\lambda(u)
	& = \left[\frac{1}{p}-\frac{1}{r}\right] \into \xi(x) |\nabla u|^p \,dx- \left[\frac{1}{1-\gamma}-\frac{1}{r}\right] \into a(x)|u|^{1-\gamma} \,dx\\
	&\geq \l[\frac{1}{p}-\frac{1}{r}\r] \xi_0 \|u\|^p- \left[\frac{1}{1-\gamma}-\frac{1}{r}\right] \into a(x)|u|^{1-\gamma} \,dx\\
	&\geq c_1 \|u\|^p-c_2\|u\|^{1-\gamma}
      \end{split}
    \end{align}
    for some $c_1,c_2 >0$, where we have used Theorem 13.17 of Hewitt-Stromberg \cite[p.\,196]{Hewitt-Stromberg-1965}, the fact that $1-\gamma<1<p$ and the Sobolev embedding theorem. From \eqref{2} it is clear that $\ph_\lambda\big|_{N_\lambda}$ is coercive.
\end{proof}

Let $m^+_\lambda=\inf_{N_\lambda^+}\ph_\lambda$.

\begin{proposition}%\label{proposition_2}
    If hypotheses H$_0$ hold, then $m^+_\lambda<0$.
\end{proposition}

\begin{proof}
    From the definition of $N^+_\lambda$, we have, for $u \in N_\lambda^+$,
    \begin{align}\label{3}
	\lambda \|u\|_r^r < \frac{p+\gamma-1}{r+\gamma-1} \into \xi(x) |\nabla u|^p\,dx.
    \end{align}
    Moreover, since $u \in N^+_\lambda \subseteq N_\lambda$, it holds
    \begin{align}\label{4}
	-\frac{1}{1-\gamma} \into a(x) |u|^{1-\gamma}\,dx= -\frac{1}{1-\gamma} \into \xi(x) |\nabla u|^p \,dx+ \frac{\lambda}{1-\gamma} \|u\|^r_r.
    \end{align}
    Applying \eqref{3}, \eqref{4}, hypotheses H$_0$ and recalling $0<\gamma<1<p<r$, we get for $u \in N^+_\lambda$
    \begin{align*}
	\ph_\lambda(u)
	& = \left[\frac{1}{p}-\frac{1}{1-\gamma}\r] \into \xi(x) |\nabla u|^p\,dx-\lambda \l[\frac{1}{r}-\frac{1}{1-\gamma}\r] \|u\|^r_r\\
	& < \left[\frac{-(p+\gamma-1)}{p(1-\gamma)} +\frac{r+\gamma-1}{r(1-\gamma)}\cdot \frac{p+\gamma-1}{r+\gamma-1}\right] \into\xi(x) |\nabla u|^p \,dx\\
	&= \frac{p+\gamma-1}{1-\gamma} \l[\frac{1}{r}-\frac{1}{p}\r]\into\xi(x)|\nabla u|^p\,dx\\
	&< 0.
    \end{align*}
    Therefore, $\ph_\lambda\big|_{N_\lambda^+}<0$ and so $m^+_\lambda<0$.
\end{proof}

\begin{proposition}\label{proposition_3}
    If hypotheses H$_0$ hold, then there exists $\lambda^*>0$ such that for all $\lambda \in (0,\lambda^*)$ we have $N^0_\lambda =\emptyset$.
\end{proposition}

\begin{proof}	
    We argue indirectly. So, suppose that  for every $\lambda^*>0$ there exists $\lambda \in (0,\lambda^*)$ such that $N^0_\lambda \neq \emptyset$. Hence, given $\lambda>0$, we can find $u\in N_\lambda$ such that
    \begin{align}\label{5}
	(p+\gamma-1)\into \xi(x) |\nabla u|^p\,dx=\lambda (r+\gamma-1)\|u\|^r_r.
    \end{align}
    Moreover, since $u \in N_\lambda$, one has
    \begin{align}\label{6}
      \begin{split}
	& (r+\gamma-1)\into \xi(x) |\nabla u|^p\,dx -(r+\gamma-1) \into a(x) |u|^{1-\gamma}\,dx\\
	& =\lambda(r+\gamma-1) \|u\|^r_r.
      \end{split}
    \end{align}
    Subtracting \eqref{5} from \eqref{6} results in
    \begin{align*}
	(r-p)\into \xi(x) |\nabla u|^p \,dx = (r+\gamma-1)\into a(x) |u|^{1-\gamma}\,dx.
    \end{align*}
    Hence, by hypotheses H$_0$,
    \begin{align*}
	(r-p)\xi_0 \|u\|^p \leq (r+\gamma-1)c_3 \|u\|^{1-\gamma}
    \end{align*}
    for some $c_3>0$. This implies
    \begin{align}\label{7}
	\|u\|^{p+\gamma-1} \leq c_4
    \end{align}
    for some $c_4>0$.
    
    On the other hand, from \eqref{5}, hypotheses H$_0$ and the Sobolev embedding theorem, we obtain
    \begin{align*}
	\|u\|^p \leq \lambda c_5 \|u\|^r
    \end{align*}
    for some $c_5>0$ and thus,
    \begin{align*}
	\l[\frac{1}{\lambda c_5}\r]^{\frac{1}{r-p}} \leq \|u\|.
    \end{align*}
    We let $\lambda\to 0^+$ and see that $\|u\|\to \infty$, contradicting \eqref{7}. Therefore, we can find $\lambda^*>0$ such that $N^0_\lambda=\emptyset$ for all $\lambda \in (0,\lambda^*)$.
\end{proof}

\begin{proposition}\label{proposition_4}
    If hypotheses H$_0$ hold, then there exists $\hat{\lambda}^*\in (0,\lambda^*]$ such that for every $\lambda \in (0,\hat{\lambda}^*)$, there exists $u^*\in N^+_\lambda$ such that
    \begin{align*}
	\ph_\lambda(u^*)=m^+_\lambda =\inf_{N^+_\lambda} \ph_\lambda
    \end{align*}
    and $u^*(x) \geq 0$ for a.\,a.\,$x\in \Omega$.
\end{proposition}

\begin{proof}
    Let $\{u_n\}_{n\geq 1} \subseteq N^+_\lambda$ be a minimizing sequence, that is,
    \begin{align}\label{8}
	\ph_\lambda(u_n) \searrow m^+_\lambda <0 \quad\text{as }n\to\infty.
    \end{align}
    Since $N^+_\lambda \subseteq N_\lambda$, from Proposition \ref{proposition_1}, we infer that
    \begin{align*}
	\{u_n\}_{n\geq 1} \subseteq \Wpzero{p} \text{ is bounded.}
    \end{align*}
    So, by passing to a suitable subsequence if necessary, we may assume that
    \begin{align}\label{9}
	u_n \weak u^* \quad \text{in } \Wpzero{p}\quad\text{and}\quad u_n\to u^* \quad\text{in } \Lp{r}.
    \end{align}
    From \eqref{8} and $u_n \weak u^*$ in $\Wpzero{p}$ we have
    \begin{align*}
            \ph_\lambda(u^*) \leq \liminf_{n \to \infty} \ph_\lambda(u_n)<0=\ph_\lambda(0).
    \end{align*}
    Hence, $u^*\neq 0$.
    
    We consider the fibering function $\psi_{u^*}\colon [0,\infty)\to \R$ defined by
    \begin{align*}
	\psi_{u^*}(t)=\ph_\lambda(tu^*) \quad\text{for all }t\geq 0.
    \end{align*}
    Moreover, let $\eta_{u^*}\colon (0,\infty)\to \R$ be the function defined by
    \begin{align*}
	\eta_{u^*}(t)=t^{p-r} \into \xi(x) |\nabla u^*|^p\,dx-t^{-\gamma -r+1}\into a(x) |u^*|^{1-\gamma}\,dx \quad\text{for all } t>0.
    \end{align*}
    Note that as $t\to 0^+$, then $\eta_{u^*}(t)\to -\infty$, since $r-p<r+\gamma-1$ and $a(x)>0$ for a.\,a.\,$x\in\Omega$, see H$_0$. Also, $\eta_{u^*}(t)\to 0$ as $t\to +\infty$  and $\eta_{u^*}(t)>0$ for
    \begin{align*}
	t> \left[ \frac{\displaystyle \into a(x) |u^*|^{1-\gamma}\,dx}{\displaystyle\into \xi(x)|\nabla u^*|^p \,dx}\right]^{\frac{1}{p+\gamma-1}}=\hat{t}>0.
    \end{align*}
    Therefore, we can find $t_0>\hat{t}$ such that
    \begin{align*}
	\eta_{u^*}(t_0)=\max_{t>0} \eta_{u^*}.
    \end{align*}
    This maximizer is unique and it is given by the solution of
    \begin{align*}
	\eta'_{u^*}(t)=0.
    \end{align*}
    Hence,
    \begin{align*}
	t_0=\l[ \frac{\displaystyle(r+\gamma-1)\into a(x) |u^*|^{1-\gamma}\,dx}{\displaystyle(r-p)\into \xi(x) |\nabla u^*|^p\,dx}\right]^{\frac{1}{p+\gamma-1}}.
    \end{align*}
    We see that
    \begin{align*}
	tu^*\in N_\lambda \quad\text{if and only if}\quad \eta_{u^*}(t)=\lambda \|u^*\|_r^r>0.
    \end{align*}
    Let $\hat{\lambda}^* \in (0,\lambda^*]$ such that
    \begin{align*}
	\eta_{u^*}(t_0)>\lambda \|u^*\|_r^r\quad\text{for all }\lambda \in (0,\hat{\lambda}^*].
    \end{align*}
    We can find $t_1<t_0<t_2$ such that
    \begin{align}\label{10}
	\eta_{u^*}(t_1)=\lambda \|u^*\|_r^r=\eta_{u^*}(t_2) \quad\text{and}\quad \eta'_{u^*}(t_2)<0<\eta'_{u^*}(t_1).   
    \end{align}
    In this proof we will only use $t_1$, we mention the existence of $t_2$ as above since it will be needed in the sequel when we will minimize over $N^-_\lambda$.
    
    Note that $\psi_{u^*}\in C^2(0,\infty)$. Therefore,
    \begin{align*}
	\psi_{u^*}'(t_1)=t_1^{p-1} \into \xi(x) |\nabla u^*|^p\,dx -t_1^{-\gamma} \into a(x) |u^*|^{1-\gamma}\,dx-\lambda t_1^{r-1} \|u^*\|^r_r,
    \end{align*}
    and
    \begin{align}\label{11}
      \begin{split}
	\psi^{''}_{u^*}(t_1)
	& =(p-1)t_1^{p-2} \into \xi(x) |\nabla u^*|^{p}\,dx +\gamma t_1^{-\gamma-1} \into a(x) |u^*|^{1-\gamma}\,dx\\
	& \quad -(r-1)\lambda t_1^{r-2} \|u^*\|_r^r.
      \end{split}
    \end{align}
    From \eqref{10} we have
    \begin{align*}
	t_1^{p-r} \into \xi(x) |\nabla u^*|^p\,dx -\lambda \|u^*\|^r_r=t_1^{-\gamma-r+1} \into a(x) |u^*|^{1-\gamma}\,dx,
    \end{align*}
    which implies that
    \begin{align}\label{12}
	t_1^{p-2} \into \xi(x) |\nabla u^*|^{p}\,dx-\lambda t_1^{r-2} \|u^*\|_r^r=t_1^{-\gamma-1} \into a(x) |u^*|^{1-\gamma}\,dx.
    \end{align}
    We will now apply \eqref{12} in \eqref{11} and obtain
    \begin{align}\label{13}
      \begin{split}
	\psi_{u^*}^{''}(t_1)
	& = [p+\gamma-1] t_1^{p-2} \into \xi(x) |\nabla u^*|^p\,dx -(r+\gamma-1)\lambda t_1^{r-2} \|u^*\|_r^r\\
	& = t_1^{-2} \l[(p+\gamma-1)t_1^p \into \xi(x) |\nabla u^*|^p \,dx - (r+\gamma-1)\lambda t_1^r \|u^*\|_r^r \r].
      \end{split}
    \end{align}
    But using \eqref{12} in \eqref{11} gives
    \begin{align}\label{14}
      \begin{split}
	&\psi^{''}_{u^*}(t_1)\\
	& = (p-1)t_1^{p-2} \into \xi(x) |\nabla u^*|^p\,dx+\gamma t_1^{-\gamma-1}\into a(x) |u^*|^{1-\gamma}\,dx\\
	& \qquad -(r-1)t_1^{r-2} \l[ t_1^{p-r} \into \xi(x) |\nabla u^*|^p\,dx-t_1^{-\gamma-r+1} \into a(x)|u^*|^{1-\gamma}\,dx\r]\\
	&=(p-r) t_1^{p-2} \into \xi(x) |\nabla u^*|^p\,dx+(r+\gamma-1) t_1^{-\gamma-1} \into a(x) |u^*|^{1-\gamma}\,dx\\
	&= t_1^{r-1} \eta'_{u^*}(t_1)>0,
      \end{split}
    \end{align}
    because of \eqref{10}.
    
    From \eqref{13} and \eqref{14} it follows that
    \begin{align*}
	(p+\gamma-1)t_1^p \into \xi(x) |\nabla u^*|^{p}\,dx -(r+\gamma-1) \lambda t_1^{r}\|u^*\|_r^r>0,
    \end{align*}
    which implies
    \begin{align}\label{15}
	t_1u^* \in N^+_\lambda, \quad \lambda \in (0,\hat{\lambda}^*].
    \end{align}
    
    Suppose that
    \begin{align}\label{16}
            \liminf_{n\to \infty} \into \xi(x) |\nabla u_n|^p\,dx> \into \xi(x)|\nabla u^*| \,dx.
    \end{align}
%     Next suppose that $u_n\not\to u^*$ in $\Wpzero{p}$. Then we must have that
%     \begin{align*}
% 	\limsup_{n\to \infty} \|\nabla u_n-\nabla u^*\|_p^p =\beta>0.
%     \end{align*}
%     By passing to a suitable subsequence if necessary, we may assume that
%     \begin{align*}
% 	\|\nabla u_n-\nabla u^*\|_p^p\to \beta>0.
%     \end{align*}
%     Hence,
%     \begin{align}\label{16}
% 	\into \xi(x) |\nabla u_n-\nabla u^*|^pdx \to \beta_0 >0.
%     \end{align}
%     Now, by applying \eqref{9}, \eqref{10}, \eqref{16} and the Brezis-Lieb lemma, see, for example, Papageorgiou-Winkert \cite[Lemma 4.1.22, p.\,292]{Papageorgiou-Winkert-2018}, we get
Applying \eqref{9}, \eqref{10} and \eqref{16} we get
    \begin{align}\label{17}
      \begin{split}
	&\liminf_{n\to\infty} \psi'_{u_n}(t_1)\\
	& =\liminf_{n\to\infty} \left[t_1^{p-1}\into \xi(x) |\nabla u_n|^p\,dx-t_1^{-\gamma} \into a(x) |u_n|^{1-\gamma}\,dx-\lambda t_1^{r-1} \|u_n\|_r^r\right]\\
	%& =\liminf_{n\to\infty} \left[t_1^{p-1}\into \xi(x) |\nabla u^*|^p\,dx+t_1^{p-1} \into \xi(x)|\nabla u_n-\nabla u^*|^p\,dx+o(1)\right.\\ 
	%&\qquad\qquad\qquad \left.-t_1^{-\gamma} \into a(x) |u_n|^{1-\gamma}dx-\lambda t_1^{r-1} \|u_n\|_r^r\right]\\
	& > t_1^{p-1} \into \xi(x) |\nabla u^*|^p\,dx-t_1^{-\gamma} \into a(x) |u^*|^{1-\gamma}\,dx-\lambda t_1^{r-1}\|u^*\|_r^r\\
	&=\psi'_{u^*}(t_1)\\
	&= t_1^{r-1} \l[\eta_{u^*}(t_1)-\lambda\|u^*\|_r^r\r]=0.
      \end{split}
    \end{align}
    From \eqref{17} we see that there exists $n_0\in\N$ such that
    \begin{align}\label{18}
	\psi'_{u_n}(t_1)>0 \quad\text{for all }n \geq n_0.
    \end{align}
    Recall that $u_n\in N^+_\lambda \subseteq N_\lambda$ and $\psi'_{u_n}(t)=t^r \eta_{u_n}(t)$. Hence
    \begin{align*}
	\psi'_{u_n}(t)<0\quad\text{for all }t\in (0,1)\quad \text{and}\quad \psi'_{u_n}(1)=0. 
    \end{align*}
    Then, by \eqref{18}, it follows $t_1>1$.
    
    Since $\psi_{u^*}$ is decreasing on $(0,t_1]$, we have
    \begin{align}\label{19}
	\ph_\lambda(t_1u^*)\leq \ph_\lambda(u^*)<m^+_\lambda.
    \end{align}
    But recall that $t_1u^*\in N^+_\lambda$ because of \eqref{15}. So, by \eqref{19}, we obtain
    \begin{align*}
	m^+_\lambda \leq \ph_\lambda(t_1u^*)<m^+_\lambda,
    \end{align*}
    a contradiction. This proves that $u_n\to u^*$ in $\Wpzero{p}$, see Papageorgiou-Winkert \cite[p.\,225]{Papageorgiou-Winkert-2018}, and so, with regards to \eqref{8},
    \begin{align*}
	\ph_\lambda(u_n)\to \ph_{\lambda}(u^*)=m^+_\lambda<0.
    \end{align*}
    
    We know that $u_n\in N^+_\lambda$ for all $n\in\N$. This implies
    \begin{align*}
	(p+\gamma-1) \into \xi(x) |\nabla u_n|^p\,dx > \lambda (r+\gamma-1)\|u_n\|^r_r \quad\text{for all }n\in\N.
    \end{align*}
    Therefore
    \begin{align}\label{20}
	(p+\gamma-1)\into \xi(x) |\nabla u^*|^p \,dx\geq \lambda (r+\gamma-1) \|u^*\|_r^r.
    \end{align}
    On account of Proposition \ref{proposition_3}, since $\lambda \in (0,\hat{\lambda}^*]$, we cannot have equality in \eqref{20}. Therefore $u^*\in N^+_\lambda$ and finally we have
    \begin{align*}
	m^+_\lambda=\ph_\lambda(u^*)\quad\text{and}\quad u^*\in N^+_\lambda.
    \end{align*}
    Since we can always replace $u^*$ by $|u^*|$, we may assume that $u^*\geq 0$ with $u^*\neq 0$.
\end{proof}

The next lemma is inspired by Lemma 3 of Sun-Wu-Long \cite{Sun-Wu-Long-2001}. In what follows we denote by $B_\eps(0)$ the open $\eps$-ball in $\Wpzero{p}$ centered at the origin, that is,
\begin{align*}
    B_\eps(0)=\l\{u\in\Wpzero{p}\colon \|u\|<\eps\r\}.
\end{align*}

\begin{lemma}\label{lemma_5}
    If hypotheses H$_0$ hold and $u\in N^+_\lambda$, then there exist $\eps>0$ and a continuous function $\vartheta\colon B_\eps(0)\to \R_+$ such that
    \begin{align*}
	\vartheta(0)=1\quad\text{and}\quad \vartheta(y)(u+y) \in N^{\pm}_\lambda \quad\text{for all }y\in B_\eps(0).
    \end{align*}
\end{lemma}

\begin{proof}
    We do the proof only for $N^+_\lambda$, the proof for $N^-_\lambda$ works in the same way. So, let $L\colon\Wpzero{p}\times (0,\infty)\to \R$ be defined by
    \begin{align*}
	L(y,t)=t^{p+\gamma-1} \into \xi(x) |\nabla (u+y)|^p \,dx-\into a(x) |u+y|^{1-\gamma}\,dx -\lambda t^{r+\gamma-1} \|u+y\|_r^r.
    \end{align*}
    Since $u\in N^+_\lambda \subseteq N_\lambda$, one has $L(0,1)=0$. Moreover, because $u \in N^+_\lambda$, it holds
    \begin{align*}
	L'_t(0,1)=(p+\gamma-1) \into \xi(x) |\nabla u|^p \,dx-\lambda (r+\gamma-1)\|u\|_r^r>0.
    \end{align*}
    Then, by the implicit function theorem, see Gasi\'{n}ski-Papageorgiou \cite[p.\,481]{Gasinski-Papageorgiou-2006}, we can find $\eps>0$ and a continuous map $\vartheta\colon B_\eps(0)\to \R_+$ such that
    \begin{align*}
	\vartheta(0)=1\quad\text{and}\quad \vartheta(y)(u+y) \in N_\lambda \quad\text{for all } y\in B_\eps(0).
    \end{align*}
    Choosing $\eps>0$ even smaller if necessary, we can have
    \begin{align*}
	\vartheta(0)=1\quad\text{and}\quad \vartheta(y)(u+y) \in N^{+}_\lambda \quad\text{for all }y\in B_\eps(0).
    \end{align*}
\end{proof}

\begin{proposition}\label{proposition_6}
    If hypotheses H$_0$ hold, $\lambda \in (0,\hat{\lambda}^*]$ and $h\in \Wpzero{p}$, then we can find $b>0$ such that
    \begin{align*}
	\ph_\lambda(u^*) \leq \ph(u^*+th)\quad\text{for all }t\in [0,b].
    \end{align*}
\end{proposition}

\begin{proof}
    We consider the function $\mu_h\colon [0,\infty)\to\R$ defined by
    \begin{align}\label{21}
      \begin{split}
	\mu_h(t)
	& =(p-1)\into \xi(x) |\nabla u^*+t\nabla h|^p\,dx\\
	& \quad +\gamma\into a(x) |u^*+th|^{1-\gamma}\,dx -\lambda(r-1) \|u^*\|_r^r.
      \end{split}
    \end{align}
    Recall that $u^*\in N^+_\lambda \subseteq N_\lambda$, see Proposition \ref{proposition_4}. Thus, we have
    \begin{align}\label{22}
	& \gamma\into \xi(x) |u^*|^{1-\gamma}\,dx=\gamma\into \xi(x) |\nabla u^*|^p\,dx-\lambda \gamma \|u^*\|_r^r
    \end{align}
    and
    \begin{align}\label{23}
	& (p+\gamma-1) \into \xi(x) |\nabla u^*|^p \,dx-\lambda (r+\gamma-1) \|u\|^r_r>0.
    \end{align}
    Combining \eqref{21}, \eqref{22} and \eqref{23} we obtain that
    \begin{align}\label{24}
	\mu_h(0)>0.
    \end{align}
    The function $\mu_h$ is continuous. So, we can find $b_0>0$ such that 
    \begin{align*}
	\mu_h(t)>0 \quad\text{for all }t\in (0,b_0),
    \end{align*}
    see \eqref{24}. Lemma \ref{lemma_5} implies that for every $t\in [0,b_0)$, we can find $\hat{\vartheta}(t)>0$ such that
    \begin{align}\label{25}
	\hat{\vartheta}(t) (u^*+th) \in N^+_\lambda\quad\text{and}\quad \hat{\vartheta}(t)\to 1 \text{ as }t\to 0^+.
    \end{align}
    Taking \eqref{25} into account we finally reach that
    \begin{align*}
	m^+_\lambda
	& = \ph_\lambda(u^*) \leq \ph_\lambda(\hat{\vartheta}(t)(u^*+th)) \quad\text{for all }t\in [0,b_0)\\
	& \leq \ph_\lambda(u^*+th) \quad\text{for all }t\in [0,b) \text{ with }b \leq b_0.
    \end{align*}
\end{proof}

The next proposition shows that $N^+_\lambda$ is a natural constraint for the functional $\ph_\lambda$, see 
Papageorgiou-R\u{a}dulescu-Repov\v{s} \cite[p.\,425]{Papageorgiou-Radulescu-Repovs-2019}.

\begin{proposition}\label{proposition_7}
    If hypotheses H$_0$ hold and $\lambda \in (0,\hat{\lambda}^*)$, then $u^*$ is a weak solution of problem \eqref{problem}.
\end{proposition}

\begin{proof}
    Let $h\in \Wpzero{p}$. From Proposition \ref{proposition_6} we know that
    \begin{align*}
	0 \leq \ph_\lambda(u^*+th)-\ph_\lambda(u^*) \quad\text{for all }0<t<b.
    \end{align*}
    This means
    \begin{align*}
	& \frac{1}{1-\gamma} \into a(x) \l[|u^*+th|^{1-\gamma}-|u^*|^{1-\gamma}\r]\,dx\\
	& \leq \frac{1}{p}\into \xi(x) \l(|\nabla (u^*+th)|^p-|\nabla u^*|^p\r)\,dx -\frac{\lambda}{r} \l[\|u^*+th\|_r^r-\|u^*\|_r^r\r].
    \end{align*}
    Multiplying by $\frac{1}{t}$ and letting $t\to 0^+$ gives 
    \begin{align*}
	\into a(x) (u^*)^{-\gamma}h\,dx \leq \into \xi(x) |\nabla u^*|^{p-2} \nabla u^* \cdot \nabla h\,dx-\lambda \into (u^*)^{r-1}h\,dx
    \end{align*}
    for all $h\in \Wpzero{p}$. Hence,
    \begin{align*}
	\into \xi(x) |\nabla u^*|^{p-2} \nabla u^* \cdot \nabla h\,dx
	= \into a(x) (u^*)^{-\gamma}h\,dx +\lambda \into (u^*)^{r-1}h\,dx
    \end{align*}
    for all $h\in \Wpzero{p}$. Thus, $u^*$ is a weak solution of \eqref{problem}.
\end{proof}

Now we are ready to generate the first positive solution of problem \eqref{problem}.

\begin{proposition}%\label{proposition_8}
    If hypotheses H$_0$ hold and $\lambda \in (0,\hat{\lambda}^*)$, then problem \eqref{problem} admits a positive solution $u^*\in \Wpzero{p}$ such that $u \in \Linf$, $u^*(x)>0$ for a.\,a.\,$x\in \Omega$ and $\ph_\lambda(u^*)<0$.
\end{proposition}

\begin{proof}
    According to Proposition \ref{proposition_4} there exists $u^*\in \Wpzero{p}$ such that 
    \begin{align*}
	u^*\in N^+_\lambda \quad\text{and}\quad  m^+_\lambda=\ph_\lambda(u^*)<0, \quad u^*\geq 0.
    \end{align*}
    From Proposition \ref{proposition_7} we know that $u^*$ is a weak solution of problem \eqref{problem}.
    
    From Giacomoni-Schindler-Tak\'{a}\v{c} \cite[Lemma A.6, p.\,142]{Giacomoni-Schindler-Takac-2007} we have that $u^*\in \Linf$. Furthermore, the Harnack inequality, see Pucci-Serrin \cite[p.\,163]{Pucci-Serrin-2007} implies that
    \begin{align*}
	u^*(x) >0 \quad\text{for a.\,a.\,}x\in\Omega.
    \end{align*}
\end{proof}

Now we start looking for a second positive solution. To this end, we will use the manifold $N^-_\lambda$.

\begin{proposition}%\label{proposition_9}
    If hypotheses H$_0$ hold, then there exists $\hat{\lambda}^*_0 \in (0,\hat{\lambda}^*]$ such that $\ph_\lambda\big|_{N^-_\lambda} \geq 0$ for all $0<\lambda\leq\hat{\lambda}_0^*$.
\end{proposition}

\begin{proof}
    Let $u\in N_\lambda$. From the definition of $N^-_\lambda$ we have
    \begin{align*}
	(p+\gamma-1) \into \xi(x) |\nabla u|^p\,dx<\lambda (r+\gamma-1) \|u\|^r_r,
    \end{align*}
    which implies
    \begin{align*}
	(p+\gamma-1) \xi_0 \|\nabla u\|^p_p<\lambda (r+\gamma-1) \|u\|^r_r.
    \end{align*}
    Then, by the embedding $\Wpzero{p}\hookrightarrow \Lp{r}$, it follows
    \begin{align*}
	(p+\gamma-1) \xi_0 c_5 \|u\|^p_r<\lambda (r+\gamma-1) \|u\|^r_r
    \end{align*}
    for some $c_5>0$. Therefore
    \begin{align}\label{26}
	\l[\frac{(p+\gamma-1)\xi_0c_5}{\lambda(r+\gamma-1)}\r]^{\frac{1}{r-p}} \leq \|u\|_r.
    \end{align}
    Suppose that the result of the proposition is not true. This means that for every $\lambda>0$ there exists $u \in N^-_\lambda$ such that $\ph_\lambda(u)<0$, that is,
    \begin{align}\label{27}
	\frac{1}{p} \into \xi(x) |\nabla u|^p\,dx -\frac{1}{1-\gamma} \into a(x)|u|^{1-\gamma}\,dx-\frac{\lambda}{r} \|u\|_r^r <0.
    \end{align}
    
    On the other hand, since $u \in N^-_\lambda\subseteq N_\lambda$, we have
    \begin{align}\label{28}
	\into \xi(x) |\nabla u|^p\,dx = \into a(x) |u|^{1-\gamma}\,dx+\lambda \|u\|_r^r.
    \end{align}
    Using \eqref{28} in \eqref{27} yields
    \begin{align*}
	\l[\frac{1}{p}-\frac{1}{1-\gamma}\r] \into a(x) |u|^{1-\gamma}\,dx+\lambda \l[\frac{1}{p}-\frac{1}{r}\right] \|u\|_r^r<0,
    \end{align*}
    which implies
    \begin{align*}
	\lambda \frac{r-p}{pr} \|u\|_r^r \leq \frac{p+\gamma-1}{p(1-\gamma)} \into a(x) |u|^{1-\gamma}\,dx
	\leq \frac{p+\gamma-1}{p(1-\gamma)} c_6 \|u\|_r^{1-\gamma}
    \end{align*}
    for some $c_6>0$. Hence
    \begin{align*}
	\|u\|_r \leq \l[\frac{(p+\gamma-1)rc_6}{\lambda(1-\gamma)(r-p)}\r]^{\frac{1}{r+\gamma-1}}
    \end{align*}
    and so
    \begin{align}\label{29}
	\|u\|_r \leq c_7 \l(\frac{1}{\lambda}\r)^{\frac{1}{r+\gamma-1}}
    \end{align}
    for some $c_7>0$.
    
    Now we use \eqref{29} in \eqref{26} and obtain
    \begin{align*}
	c_8 \l(\frac{1}{\lambda}\r)^{\frac{1}{r-p}} \leq c_7 \l(\frac{1}{\lambda}\r)^{\frac{1}{r+\gamma-1}} \quad\text{with}\quad c_8=\l[\frac{(p+\gamma-1)\xi_0}{r+\gamma-1}\r]^{\frac{1}{r-p}}>0.
    \end{align*}
    This implies
    \begin{align*}
	c_9 \leq \lambda^{\displaystyle\frac{p+\gamma-1}{\displaystyle(r+\gamma-1)(r-p)}}\quad\text{with}\quad c_9=\frac{c_8}{c_7}>0.
    \end{align*}
    Letting $\lambda \to 0^+$ leads to a contradiction. So, we can find $0<\hat{\lambda}_0^*\leq \hat{\lambda}^*$ such that $\ph_\lambda\big|_{N^-_\lambda}\geq 0$ for all $\lambda \in (0,\hat{\lambda}_0^*]$.
\end{proof}

Now we minimize $\ph_\lambda$ on the manifold $N^-_\lambda$.

\begin{proposition}%\label{proposition_10}
    If hypotheses H$_0$ hold and $\lambda \in (0,\lambda^*_0)$, then we can find $v^*\in N^-_\lambda$ with $v^*\geq 0$ such that
    \begin{align*}
	m^-_\lambda=\inf_{N^-_\lambda} \ph_\lambda=\ph_\lambda(v^*).
    \end{align*}
\end{proposition}

\begin{proof}
    The proof of the proposition is the same as that of Proposition \ref{proposition_4}. Only now as we already hinted in that proof, we use the point $t_2>t_0$ for which we have
    \begin{align*}
	\eta_{v^*}(t_2)=\lambda \|v^*\|_r^r \quad\text{and}\quad \eta'_{v^*}(t_2)<0,
    \end{align*}
    see \eqref{10}. Then we conclude that
    \begin{align*}
	v^*\in N^-_\lambda, \quad v^* \geq 0, \quad m^-_\lambda =\ph_\lambda(v^*).
    \end{align*}
\end{proof}

Applying Lemma \ref{lemma_5} and reasoning as in the proofs of Propositions \ref{proposition_6} and \ref{proposition_7} we show that $N^-_\lambda$ is a natural constraint for the energy functional $\ph_\lambda$ as well.

\begin{proposition}%\label{proposition_11}
    If hypotheses H$_0$ hold and $\lambda \in (0,\hat{\lambda}_0^*)$, then $v^*$ is a weak solution of problem \eqref{problem}.
\end{proposition}

Therefore, we have a second positive solution $v^*\in \Wpzero{p}\cap \Linf$ and by Harnack's inequality we have $v^*(x) >0$ for a.\,a.\,$x\in\Omega$.

Finally, we can state the following multiplicity theorem for problem \eqref{problem}.

\begin{theorem}
    If hypotheses H$_0$ hold, then there exists $\hat{\lambda}^*_0>0$ such that for all $\lambda \in (0,\hat{\lambda}^*_0)$, problem \eqref{problem} has at least two positive solutions
    \begin{align*}
	u^*,v^*\in\Wpzero{p}\cap \Linf, \quad u^*(x)>0, \, v^*(x)>0 \text{ for a.\,a.\,}x\in \Omega
    \end{align*}
    and
    \begin{align*}
	\ph_\lambda(u^*) <0 < \ph_\lambda(v^*).
    \end{align*}
\end{theorem}

\begin{remark}
    It is an interesting open problem whether the multiplicity theorem above holds if we assume that
    \begin{align*}
	\xi\in\Linf \quad\text{and}\quad \xi(x)>0 \text{ for a.\,a.\,}x\in\Omega,
    \end{align*}
    but not necessarily bounded away from zero.
\end{remark}

\section*{Acknowledgment}

The authors wish to thank a knowledgeable referee for his/her corrections and remarks.

The second author thanks the National Technical University of Athens for the kind hospitality during a research stay in June 2019.


\begin{thebibliography}{99}

\bibitem{Gasinski-Papageorgiou-2016}
    L.\,Gasi{\'n}ski, N.\,S.\,Papageorgiou,   
    ``Exercises in Analysis. {P}art 2: Nonlinear Analysis'', Springer, Heidelberg, 2016.
    
\bibitem{Gasinski-Papageorgiou-2006}
    L.\,Gasi\'{n}ski, N.\,S.\,Papageorgiou,
    ``Nonlinear Analysis'',
    Chapman \& Hall/CRC, Boca Raton, FL, 2006.

\bibitem{Giacomoni-Schindler-Takac-2007}
    J.\,Giacomoni, I.\,Schindler, P.\,Tak\'{a}\v{c},
    {\it Sobolev versus {H}\"{o}lder local minimizers and existence of multiple solutions for a singular quasilinear equation},
    Ann.\,Sc.\,Norm.\,Super.\,Pisa Cl.\,Sci.\,(5) {\bf 6} (2007), no.\,1, 117--158.

\bibitem{Haitao-2003}
    Y.\,Haitao,
    {\it Multiplicity and asymptotic behavior of positive solutions for a singular semilinear elliptic problem},
    J.\,Differential Equations {\bf 189} (2003), no.\,2, 487--512.
   
\bibitem{Hewitt-Stromberg-1965}
    E.\,Hewitt, K.\,Stromberg,
    ``Real and Abstract Analysis'',
    Springer-Verlag, New York, 1965.
    
\bibitem{Leonardi-Papageorgiou-2019a}
    S.\,Leonardi, N.\,S.\,Papageorgiou,
    {\it Positive solutions for nonlinear Robin problems with indefinite potential and competing nonlinearities},
    Positivity, https://doi.org/10.1007/s11117-019-00681-5.
    
\bibitem{Leonardi-Papageorgiou-2019b}
    S.\,Leonardi, N.\,S.\,Papageorgiou,
    {\it On a class of critical Robin problems},
    Forum Math., https://doi.org/10.1515/forum-2019-0160.    
    
\bibitem{Lieberman-1991}
    G.\,M.\,Lieberman,
    {\it The natural generalization of the natural conditions of {L}adyzhenskaya and {U}ral\cprime tseva for elliptic equations},
    Comm.\,Partial Differential Equations {\bf 16} (1991), no.\,2-3, 311--361.
         
\bibitem{Papageorgiou-Radulescu-Repovs-2019}
    N.\,S.\,Papageorgiou, V.\,D.\,R\u{a}dulescu, D.\,D.\,Repov\v{s},
    ``Nonlinear Analysis --- Theory and Methods'',
    Springer, Cham, 2019.
      
\bibitem{Papageorgiou-Smyrlis-2015}
    N.\,S.\,Papageorgiou, G.\,Smyrlis,
    {\it A bifurcation-type theorem for singular nonlinear elliptic equations},
    Methods Appl.\,Anal.\,{\bf 22} (2015), no.\,2, 147--170.
   
\bibitem{Papageorgiou-Winkert-2018}
    N.\,S.\,Papageorgiou, P.\,Winkert,
    ``Applied Nonlinear Functional Analysis. An Introduction'',
    De Gruyter, Berlin, 2018.
    
\bibitem{Papageorgiou-Winkert-2019}
    N.\,S.\,Papageorgiou, P.\,Winkert,
    {\it Singular {$p$}-{L}aplacian equations with superlinear perturbation},
    J.\,Differential Equations {\bf 266} (2019), no.\,2-3, 1462--1487.
  
\bibitem{Perera-Zhang-2005}
    K.\,Perera, Z.\,Zhang,
    {Multiple positive solutions of singular {$p$}-{L}aplacian problems by variational methods},
    Bound.\,Value Probl.\,{\bf 2005}, no.\,3,  377--382.

\bibitem{Pucci-Serrin-2007}
    P.\,Pucci, J.\,Serrin,
    ``The Maximum Principle'',
    Birkh\"auser Verlag, Basel, 2007.
    
\bibitem{Sun-Wu-Long-2001}
    Y.\,Sun, S.\,Wu, Y.\,Long,
    {Combined effects of singular and superlinear nonlinearities in some singular boundary value problems},
    J.\,Differential Equations {\bf 176} (2001), no.\,2, 511--531.



\end{thebibliography}
\end{document}